\documentclass[12pt]{amsart}
\usepackage{amsmath,amsthm,amsfonts,amssymb,eucal}

\renewcommand {\a}{ \alpha }

\newcommand{\G}{\Gamma}

\renewcommand{\L}{\Lambda}
\newcommand{\z}{\zeta}

\newcommand{\p}{\partial}

\newcommand{\Om}{\Omega}

\newcommand{\oq}{\ {\raise 7pt\hbox{${\scriptstyle\circ}$}}
\kern -7pt{
\hbox{$Q$}}}

\newcommand{\R}{ \mathbb R}

\newcommand {\GA}{\mathfrak A}

\newcommand {\GS}{\mathfrak S}

\newcommand {\GW}{\mathfrak W}

\newcommand {\BS}{\mathbf S}

\newcommand {\bx}{\mathbf x}

\newcommand {\bz}{\mathbf z}
\newcommand {\by}{\mathbf y}

\newcommand {\bn}{\mathbf n}

\newcommand{\SN}{{\sf{N}}}

\newcommand {\bmu}{\boldsymbol\mu}

\newcommand {\bxi}{\boldsymbol\xi}

\newcommand{\plainC}[1]{\textup{{\textsf{C}}}^{#1}}

\DeclareMathOperator*{\esssup}{ess-sup}

\DeclareMathOperator{\tr}{{tr}}

\newcommand{\1}
{{\,\vrule depth3pt height9pt}{\vrule depth3pt height9pt}
{\vrule depth3pt height9pt}{\vrule depth3pt height9pt}\,}

\DeclareMathOperator {\re} {{Re}}

\DeclareMathOperator{\op}{{Op}}

\DeclareMathOperator{\supp}{{supp}}

\newtheorem{thm}{Theorem}[section]
\newtheorem{cor}[thm]{Corollary}
\newtheorem{lem}[thm]{Lemma}
\newtheorem{prop}[thm]{Proposition}

\theoremstyle{definition}
\newtheorem{defn}[thm]{Definition}

\newtheorem{rem}[thm]{Remark}

\numberwithin{equation}{section}

\newcommand{\bee}{\begin{equation}}
\newcommand{\ene}{\end{equation}}
\newcommand{\bees}{\begin{equation*}}
\newcommand{\enes}{\end{equation*}}
\newcommand{\bes}{\begin{split}}
\newcommand{\ens}{\end{split}}

\newcommand{\bet}{\begin{thm}}
\newcommand{\ent}{\end{thm}}
\newcommand{\bel}{\begin{lem}}
\newcommand{\enl}{\end{lem}}
\newcommand{\bec}{\begin{cor}}
\newcommand{\enc}{\end{cor}}
\newcommand{\bep}{\begin{proof}}
\newcommand{\enp}{\end{proof}}
\newcommand{\ber}{\begin{rem}}
\newcommand{\enr}{\end{rem}}

\begin{document}
\hoffset -4pc

\title
[The Widom conjecture]
{{ 
Wiener-Hopf operators in higher dimensions: 
the Widom conjecture for piece-wise smooth domains}}
\author{A.V. Sobolev}
\address{Department of Mathematics\\ University College London\\
Gower Street\\ London\\ WC1E 6BT UK}
\email{a.sobolev@ucl.ac.uk}
\keywords{Wiener-Hopf operators, 
pseudo-differential operators with discontinuous symbols,
quasi-classical asymptotics}
\subjclass[2010]{Primary  47G30, 35S05; Secondary 45M05, 47B10, 47B35}
\date{\today}

\begin{abstract}
We prove a two-term quasi-classical trace asymptotic formula for the 
functions of multi-dimensional Wiener-Hopf operators with discontinuous 
symbols. The discontinuities occur on surfaces which are assumed to be 
piece-wise smooth. Such a two-term formula was conjectured by H. Widom in 1982, and 
proved by A. V. Sobolev for smooth surfaces in 2009. 

\end{abstract}

\maketitle
  
\section{Introduction}
  
The quasi-classical functional 
calculus for smooth pseudo-differential operators was 
developed more than three decades ago (see e.g. \cite{Rob} ) and now it is 
considered a standard tool of microlocal analysis and spectral theory. 
On the contrary, for pseudo-differential operators 
with discontinuous symbols results are sparse and less well known.  
Various quasi-classical trace type formulas for Wiener-Hopf operators 
were obtained by  H. Widom in the 80's. 
In this article we shall be concerned with a multi-dimensional 
generalisation of one such result which has become known as 
\textsl{The Widom Conjecture}.  
Let  $a = a(\bx, \bxi), \bx, \bxi\in\R^d, d\ge 1$ be a smooth symbol.  
Introduce the standard notation for the left and right 
pseudo-differential operators with symbol $a$ and a 
quasi-classical parameter $\a>0$:
\begin{equation}\label{pdoleft:eq}
(\op^{\rm l}_\a (a) u)(\bx) = \biggl(\frac{ \a}{2\pi}\biggr)^d
\iint e^{i\a(\bx-\by)\bxi} a(\bx, \bxi) u(\by) d\bxi d\by,
\end{equation}
\begin{equation}\label{pdoright:eq}
(\op^{\rm r}_\a (a) u)(\bx) = \biggl(\frac{\a}{2\pi}\biggr)^d
\iint e^{i\a(\bx-\by)\bxi} a(\by, \bxi) u(\by) d\bxi d\by,
\end{equation}
for any function $u$ from the Schwartz class on $\R^d$.  
If the function $a$ depends only on $\bxi$ then the operators
$\op_\a^{\rm l}(a), \op_\a^{\rm r}(a)$  coincide with each other,
and we simply write $\op_\a(a)$. 
Here
and below integrals without indication of the domain  
are assumed to be taken over the entire 
Euclidean space $\R^d$. 

Let $\L, \Om$ be bounded domains in $\R^d$, and let $\chi_\L, \chi_\Om$ 
be their characteristic functions,\ $P_{\Om, \a} = \op_\a(\chi_\Om)$. 
We are interested in spectral properties 
of the operators 
\begin{equation*}
T_\a(a) = T_\a(a; \L, \Om) = \chi_{\L} P_{\Om, \a}\op^{\rm l}_\a(a) 
P_{\Om, \a}\chi_{\L},
\end{equation*}
and  
\begin{equation*}
S_\a(a) = S_\a(a; \L, \Om) = \chi_{\L} P_{\Om, \a}\ 
\re\op^{\rm l}_\a(a)\  P_{\Om, \a}\chi_{\L}.
\end{equation*}
These operators are naturally interpreted as multi-dimensional 
Wiener-Hopf operators with discontinuous symbols. 
Our aim is to find asymptotic formulas for 
the traces of the form $\tr g(T_\a), \tr g(S_\a)$ 
as $\a\to\infty$, with suitable functions $g$, $g(0) = 0$.   
If one of the domains, e.g. $\Om$, coincides with $\R^d$ then assuming that 
$a$ is infinitely differentiable and decays sufficiently fast, one can 
write out complete asymptotic expansions of the above traces in powers of $\a^{-1}$, 
see \cite{Widom_85}. Our main focus will be on the case when both domains 
$\L$ and $\Om$ are distinct from $\R^d$.
%
The Widom Conjecture states 
(see \cite{Widom_82}) that in this case 
\begin{equation}\label{widom0:eq}
\tr g(T_\a) = \a^d \ \GW_0 + \a^{d-1} \log\a \ \GW_1 + o(\a^{d-1}\log\a), 
\end{equation} 
as $\a\to\infty$. 
The precise formulas for the coefficients $\GW_0$, $\GW_1$ 
are given in Sect. \ref{main:sect}. The first term 
in \eqref{widom0:eq} 
is the standard Weyl asymptotics, whereas the second term is non-standard, 
and it 
describes the contribution of the boundaries $\p\L$, $\p\Om$. 
Emphasise that the second term contains a 
$\log$-factor which makes it different from 
the familiar asymptotic expansion in powers of $\a^{-1}$. 
The formula \eqref{widom0:eq} was proved by 
H. Widom in \cite{Widom_82} for $d=1$. For $d\ge 2$, 
in the case when one of the 
domains is a half-space, \eqref{widom0:eq} was justified in \cite{Widom_90}. 
For arbitrary bounded smooth 
domains in $\R^d$, $d\ge 2$, the conjecture was proved in \cite{Sob}. 

The main aim of this paper 
is to extend \eqref{widom0:eq} to piece-wise smooth domains. 
Apart from the purely mathematical motivation, the interest in such domains 
is dictated by applications in Mathematical Physics, and 
in particular in Quantum Information Theory, see \cite{GiKl}, \cite{HLS}, 
\cite{LSS}.  
The proof is based on the papers \cite{Sob}, \cite{Sob_2013}. 
Using a convenient partition of unity one separates contributions 
from the smooth and non-smooth parts of the boundaries $\p\L$ and $\p\Om$. 
For the smooth part one applies directly the local version 
of the asymptotic formula of the form \eqref{widom0:eq} 
from \cite{Sob}, 
whereas for the non-smooth part it suffices 
to establish appropriate trace bounds.  
Here a key role is played by inequalities obtained for arbitrary 
Lipschitz  domains in \cite{Sob_2013}. 
As a result one 
checks that the non-smooth portion of the boundaries contributes 
a term of size $o(\a^{d-1}\log\a)$, which leads to the global asymptotics 
\eqref{widom0:eq}.

The author is grateful to J. Oldfield and W. Spitzer for critical remarks. 
This work was supported by EPSRC grant EP/J016829/1.

\section{Main results}\label{main:sect}  
  
We begin with describing the classes of domains with which we work.   
In what follows we always assume that $d\ge 2$.

\begin{defn}
\begin{enumerate}
\item
We say that $\L\subset\R^d$ is a  
\textit{basic Lipschitz domain} 
(resp. \textit{basic $\plainC{m}$-domain, $m=1, 2, \dots$}) 
if  
there exists a Lipschitz (resp. $\plainC{m}$-) function 
$\Phi = \Phi(\hat\bx),\ \hat\bx\in\R^{d-1},$ such that 
with a suitable choice of the Cartesian coordinates $\bx = (\hat\bx, x_d)$, 
$\hat\bx = (x_1, x_2, \dots, x_{d-1})$ the domain $\L$ 
is represented as
\begin{equation}\label{Lambda:eq}
 \L = \{\bx\in\R^d: x_d > \Phi(\hat\bx)\}.  
 \end{equation}
 For a basic Lipschitz domain 
 the function $\Phi$ is assumed to be \textsl{uniformly} Lipschitz, 
 i.e. the  constant  
\begin{equation}\label{Mphi:eq}
 M = M_{\Phi} = \sup_{\substack{\hat\bx, \hat\by,\\
 \hat\bx\not = \hat\by 
} } 
 \frac{|\Phi(\hat\bx) - \Phi(\hat\by)|}
 {|\hat\bx - \hat\by|}
 \end{equation}
 is finite. 
For a basic $\plainC{m}$-domain  all the derivatives 
$\nabla^n\Phi$, $n = 1, 2, \dots, m$, are assumed to be uniformly bounded 
on $\R^{d-1}$.  
For a basic domain we use the notation $\L = \G(\Phi)$. 

\item 
A domain $\L\subset\R^d$ is said to be 
\textit{Lipschitz} (\textit{resp. $\plainC{m}, m = 1, 2, \dots$}) if $\L\not=\R^d$ and 
locally it can be represented by 
basic Lipschitz (resp. $\plainC{m}$-) domains, i.e.  
for any $\bz\in\L$ there is a radius $r >0$ such that 
$B(\bz, r)\cap \L = B(\bz, r)\cap \L_0$ with some basic 
Lipschitz (resp. $\plainC{m}$-) 
domain $\L_0 = \L_0(\bz)$ or with $\L_0 = \R^d$. 
In this case the boundary $\p\L$ is said to be a $(d-1)$-dimensional 
Lipschitz (resp. $\plainC{m}$-) surface. 

\item 
A basic Lipschitz 
domain $\L=\G(\Phi)$ is said to be piece-wise $\plainC{m}$  
with some $m = 1, 2, \dots$, if  the function 
$\Phi$ is $\plainC{m}$-smooth away from a collection of 
finitely many  
$(d-2)$-dimensional Lipschitz surfaces $L_1, L_2, \dots, L_n\subset \R^{d-1}$. 
More precisely, if $\L$ is given by \eqref{Lambda:eq} then 
for any open ball $B\subset\R^{d-1}$ such that $\overline B$ is disjoint 
with all surfaces $L_j, j=1, 2, \dots, n,$ we have $\Phi\in \plainC{m}(\overline B)$. 
Note that the derivatives of $\Phi$ are not required to be bounded 
uniformly in the choice of the ball $B$. 
We denote 
\begin{equation*}
(\p\L)_{\rm s} = \biggl\{\bx = \bigl(\hat\bx, \Phi(\hat\bx)\bigr), \hat\bx\in \bigcup_{j=1}^n L_j\biggr\},
\end{equation*}
i.e. $(\p\L)_{\rm s}\subset\p\L$ is the set of points 
where the $\plainC{m}$-smoothness of the surface $\p\L$ may break down. 

\item 
A Lipschitz domain $\L$ 
is said to be piece-wise  $\plainC{m}$, $m = 1, 2, \dots$,  
if locally it can be represented by piece-wise $\plainC{m}$ basic domains. 
As for the basic domains, by $(\p\L)_{\rm s}\subset\p\L$ we denote the set 
of points where the $\plainC{m}$-smoothness of $\p\L$ may break down. 
\end{enumerate}
\end{defn}

Let us define the asymptotic coefficients entering the main asymptotic formulas. 
For a symbol $b=b(\bx, \bxi)$ let 
\begin{equation}\label{w0:eq}
\GW_0(b) = \GW_0(b; \L, \Om)
= \frac{1}{(2\pi)^d} \int_{\L}\int_{\Om} b(\bx, \bxi) d\bxi d\bx.
\end{equation}
For any $(d-1)$-dimensional Lipschitz surfaces $L, P$ denote
\begin{equation}\label{w1:eq}
\GW_1(b) = \GW_1(b ; L, P) = \frac{1}{(2\pi)^{d-1}}\int_{L} \int_{P} b(\bx, \bxi)
| \bn_{L}(\bx)\cdot\bn_{P}(\bxi) |  dS_{\bxi} dS_{\bx},
\end{equation}
where $\bn_{L}(\bx)$ and $\bn_P(\bxi)$ denote the exterior unit normals to
$L$ and $P$ defined for a.a. $\bx$ and $\bxi$ respectively. 
For any continuous function $g$ 
on $\mathbb C$ such that $g(0) = 0$, and any number $s\in\mathbb C$, 
we also define
\begin{equation}\label{GA:eq}
\GA(g; s) = \frac{1}{(2\pi)^2}\int_0^1 \frac{g(st) - t g(s)}{t(1-t)} dt.
\end{equation}
The next theorem contains the main result of the paper.  

\begin{thm}\label{main_anal:thm}
Let $\L, \Om\subset\R^d$, $d\ge 2$ be bounded 
Lipschitz domains in $\R^d$ such that 
$\L$ is piece-wise $\plainC1$  and $\Om$ is piece-wise $\plainC3$.
Let  $a = a(\bx, \bxi)$ be a symbol whose distributional derivatives 
satisfy the bounds  
\begin{equation}\label{dersym:eq}
\max_{\substack{0\le n\le d+2\\ 0\le m\le d+2}}
\esssup_{\bx, \bxi}
|\nabla_{\bx}^n \nabla_{\bxi}^m a(\bx, \bxi)|
< \infty.
\end{equation}
Let $g$ be a function on $\mathbb C$ 
such that $g(0) = 0$, 
analytic in a disk of sufficiently large radius. 
Then 
\begin{align}\label{main_anal:eq}
\tr  g(T_\a(a))
= &\ \a^d \GW_0(g(a); \L, \Om)\notag\\[0.2cm]
&\ + \a^{d-1} \log\a\  \GW_1(\GA(g; a); \p\L, \p\Om) + o(\a^{d-1}\log\a),
\end{align}
as $\a\to\infty$.
\end{thm}

For the self-adjoint operator $S_\a(a)$ we have a wider choice of
functions $g$:

\begin{thm}\label{main_s:thm}
Let the domains $\L, \Om\subset\R^d$, $d\ge 2$, and the symbol $a$ be 
as in Theorem \ref{main_anal:thm}.  
Then for any function $g\in\plainC\infty(\R)$, such that
$g(0) = 0$, one has
\begin{align}\label{main_s:eq}
\tr  g(S_\a(a))
= &\ \a^d \GW_0(g(\re a); \L, \Om)\notag\\[0.2cm]
&\ + \a^{d-1} \log\a\  \GW_1(\GA(g; \re a); \p\L, \p\Om) + o(\a^{d-1}\log\a),
\end{align}
as $\a\to\infty$.
\end{thm}

As in \cite{Sob} the crucial step of the proof is 
to prove the formula \eqref{main_anal:eq} for polynomial functions. 

\begin{thm}\label{main:thm}
Let the domains $\L, \Om\subset\R^d$, $d\ge 2$,
 and the symbol $a$ be as in Theorem \ref{main_anal:thm}. 
Then for $g_p(t) = t^p, p= 1, 2,  \dots, $ we have
\begin{align}\label{main:eq}
\tr  g_p(T_\a(a))
= &\ \a^d \GW_0(g_p(a); \L, \Om)\notag\\[0.2cm]
&\ + \a^{d-1} \log\a\  \GW_1(\GA(g_p; a); \p\L, \p\Om) + o(\a^{d-1}\log\a),
\end{align}
as $\a\to\infty$.
If $T_\a(a)$ is replaced with $S_\a(a)$, then the same formula holds
with the symbol $a$ replaced by $\re a$ on the right-hand side.
\end{thm}

In the next theorem 
the domain $\L$ is allowed to be unbounded, 
in which case we replace 
formula \eqref{main:eq} 
with its regularized variant.

\begin{thm}\label{main_inf:thm}
Let $\L, \Om\subset\R^d$, $d\ge 2$ be Lipshitz 
domains in $\R^d$ such that 
\begin{enumerate}
\item
$\Om$ is bounded and piece-wise $\plainC3$,
\item
$\L$ or $\R^d\setminus\L$ is bounded, and  
$\L$ is piece-wise $\plainC1$. 
\end{enumerate}
Let the symbol $a$ be as in Theorem \ref{main_anal:thm}. Then 
\begin{align}\label{main_inf:eq}
\lim_{\a\to\infty}\frac{1}{\a^{d-1}\log\a}
\tr\bigl[
g_p(T_\a(a; \L, \Om)) - \chi_\L g_p(T_\a(a; &\ \R^d, \Om))\chi_\L 
\bigr]\notag\\[0.2cm]
=  &\ \GW_1(\GA(g_p; a); \p\L, \p\Om),
\end{align} 
for any $p = 1, 2, \dots$. 
If $T_\a(a)$ is replaced with $S_\a(a)$, then the same formula holds
with the symbol $a$ replaced by $\re a$ on the right-hand side.
\end{thm}

Note that for bounded domains $\L$ formula \eqref{main_inf:eq} 
is just another way to write the asymptotics \eqref{main:eq}, 
see Proof of Theorem \ref{main_inf:thm}. On the other hand, for 
unbounded $\L$ formula \eqref{main_inf:eq} is an independent result. 

Theorems \ref{main_anal:thm} and \ref{main_s:thm} are derived from Theorem 
\ref{main:thm} in the same way as in \cite{Sob} for smooth domains, 
and we do not provide details.
However the methods of \cite{Sob} do not allow one to 
derive from Theorem \ref{main_inf:thm} 
analogues of Theorems \ref{main_anal:thm} or \ref{main_s:thm} 
for unbounded domains $\L$. This generalization will be done in another 
publication. 

The main focus of the rest of this paper is on the proof 
of Theorems \ref{main:thm} and \ref{main_inf:thm}. 


 \section{Auxiliary results}

Here we collect some trace  estimates and asymptotic formulas 
from \cite{Sob} and \cite{Sob_2013} used in the proofs.  
The trace estimates 
established in \cite{Sob} required that $\L$ and $\Om$ be 
$\plainC1$-smooth domains. 
In \cite{Sob_2013} most of those estimates are proved under the Lipschitz 
assumption only. 
On the other hand, the article \cite{Sob_2013} does not duplicate 
\cite{Sob}, and thus in the current article some of the estimates from \cite{Sob} 
are re-proved for Lipschitz domains.

\subsection{Notation. Smooth symbols}
In order to allow consideration of symbols $b = b(\bx, \bxi)$ 
with different scaling properties, we define for any $\ell, \rho>0$ 
the norms
\begin{equation}\label{norm:eq}
\SN^{(n, m)}(b; \ell, \rho)
= \underset{\substack
{0\le k\le n\\
0\le r\le m}}
\max \ \underset{\bx, \bxi} 
{\textup{ess}\sup} \ \ell^{k} \rho^{r}
|\nabla_{\bx}^{k} \nabla_{\bxi}^r b(\bx, \bxi)|,
\end{equation}
with $n, m = 0, 1, \dots$. 
If the norm \eqref{norm:eq} is finite for some (and hence for 
all) $\ell, \rho>0$ then 
we say that the symbol $b$ belongs to the class $\BS^{(n, m)}$. 
 
Below we often assume that various symbols $b = b(\bx, \bxi)$ 
are compactly supported, and the choice of the parameters $\ell, \rho$ 
in \eqref{norm:eq} is coordinated with the size of support. Precisely, we suppose that 
\begin{equation}\label{support:eq}
b \ \textup{is supported on}\ \ 
B(\bz, \ell)\times B(\bmu, \rho),  
\end{equation}
with some $\bz, \bmu\in\R^d$.

In what follows most of the bounds are obtained under the assumption that 
$\a\ell\rho\ge \ell_0$ with some fixed positive number $\ell_0$. 
The constants featuring in all the estimates below are independent of the 
symbols involved as well as of the parameters  
$\bz, \bmu, \a, \ell, \rho$ but may depend 
on the constant $\ell_0$. 

We begin with some natural estimates for smooth symbols. The notation $\GS_1$ 
is used for the trace class, and $\|\ \cdot\ \|_{\GS_1}$ -- for the trace class norm.

\begin{prop}\label{product:prop}
Let $a, b\in\BS^{(d+1, d+2)}$ 
be some symbols, and suppose that 
$b$ satisfies \eqref{support:eq}. 
Assume that $\a\ell\rho\ge \ell_0$. 
Then for $k = [d/2]+1$:
\begin{gather*}
\|\op_{\a}^{\rm l}(a) \| + 
\|\op_{\a}^{\rm r}(a)\|\le C \SN^{(k, d+1)}(a, \ell, \rho), \\[0.2cm]
\| \op^{\rm l}_\a(a) - \op^{\rm r}_\a(a)\|
\le C(\a\ell\rho)^{-1}\SN^{(k, d+2)}(a, \ell, \rho),
\end{gather*}
and 
\begin{gather}
\|\op^{\rm l}_\a(b)\|_{\GS_1}\le C(\a\ell\rho)^d 
\SN^{(d+1, d+1)}(b; \ell, \rho),\notag\\[0.2cm]
\|\op_\a^{\rm l}(b) - \op_\a^{\rm r}(b)\|_{\GS_1}\le C(\a\ell\rho)^{d-1}
\SN^{(d+1, d+2)}(b; \ell, \rho),\label{lr:eq}\\[0.2cm]
\| \op_\a^{\rm l}(a) \op_\a^{\rm l}(b) - \op_\a^{\rm l}(ab)\|_{\GS_1}\le
C (\a\ell\rho)^{d-1} \SN^{(d+1, d+2)}(a; \ell, \rho)
\SN^{(d+1, d+2)}(b; \ell, \rho).\notag
\end{gather}
\end{prop}

The boundedness of the operators $\op^{\rm l}_\a$, $\op^{\rm r}_\a$ 
is a classical fact, and it can be found, e.g. in \cite{Cordes},
Theorem $B_1'$, where it was established under somewhat 
weaker smoothness assumptions. For the other estimates 
see \cite{Sob}, Lemmas 3.10--3.12 and Corollary 3.13.

\subsection{Bounds for basic domains} 
Theorems 
\ref{main:thm} and \ref{main_inf:thm} 
will be deduced from the asymptotics of 
``local" traces of the form $\tr\bigl(\op^{\rm l}_\a(b) g_p(T_\a(a))\bigr)$,  
$g_p(t) = t^p$, $p=1, 2, \dots$, 
with a compactly supported symbol $b$. 
In this section we concentrate on such ``localized" operators. 
In fact, due to the bound \eqref{lr:eq} it will be unimportant which 
of the operators $\op_\a^{\rm l}(b)$ or $\op_\a^{\rm r}(b)$ is used for 
this localization. Thus we often use the notation $\op_\a(b)$ to denote any of these 
two operators.   

First we obtain some bounds for the case when both domains $\L$ and $\Om$ 
are basic Lipschitz, i.e. $\L = \G(\Phi)$ and $\Om = \G(\Psi)$ with some 
uniformly Lipschitz functions $\Phi$ and $\Psi$. The choice of Cartesian 
coordinates for which $\L$ or $\Om$ have the form \eqref{Lambda:eq} 
is not assumed to be the same for both domains. 
The constants in the estimates below depend only on the 
Lipschitz constants $M_\Phi$, $M_\Psi$ for the functions $\Phi$ and $\Psi$, 
and not on any other properties of the domains.   

First one needs the following commutator estimates.

\begin{prop}\label{sandwich:prop} 
Let $\L$, $\Om$ be basic Lipschitz domains. 
Let the symbol $b\in\BS^{(d+2, d+2)}$ satisfy \eqref{support:eq}. 
Assume that $\a\ell\rho\ge \ell_0$. Then
\begin{equation*}
\|[\op_\a(b), P_{\Om, \a}]\|_{\GS_1} +
\|[\op_\a(b), \chi_\L]\|_{\GS_1}\le C  (\a\ell\rho)^{d-1} 
\SN^{(d+2, d+2)}(b; \ell, \rho).
\end{equation*}
\end{prop}

See \cite{Sob_2013}, Remark 4.3. 

Using these commutator estimates we can now reduce the problem
to the operator $T_\a(1)$. 

\begin{lem}\label{through:lem}
Let 
each of the domains $\L$ and $\Om$ be either 
a basic Lipschitz domain, or $\R^d$. 
Let $a, b\in\BS^{(d+2, d+2)}$,  
and assume that $b$ satisfies 
\eqref{support:eq}. 
Let $\a\ell\rho\ge \ell_0$. 
Then 
\begin{align}\label{through:eq}
\| \op_\a(b) g_p\bigl(T_\a(a)\bigr) 
- &\ \op_\a(a^p b) g_p\bigl(T_\a(1)\bigr)\|_{\GS_1}\notag\\[0.2cm]
\le &\ 
 C_p(\a\ell\rho)^{d-1} 
\SN^{(d+2, d+2)}(b; \ell, \rho)\bigl(\SN^{(d+2, d+2)}(a; \ell, \rho)\bigr)^p,
\end{align}
for any $p = 1, 2, \dots$. 
The same bound holds if one replaces $T_\a$ with $S_\a$. 
\end{lem}

\begin{proof}
Without loss of generality assume that $\bz = \bmu = \bold0$ and 
that the $\SN$-norms on the right-hand side of \eqref{through:eq} equal $1$. 
Let $\z, \eta\in 
\plainC\infty_0(\R^{d})$ 
be functions $\z = \z(\bx), \eta = \eta(\bxi)$ supported in the  balls 
$B(\bold0, 2\ell)$ and $B(\bold0, 2\rho)$ respectively 
such that $b \z \eta = b$, 
and such that 
\begin{equation*}
\ell^n |\nabla_\bx^n\z(\bx)| + 
\rho^n|\nabla_{\bxi}^n\eta(\bxi)|\le \tilde C_n,\ n = 0, 1, \dots.
\end{equation*}
Represent $b = b (\z\eta)^p$ and commute the symbol $(\z(\bx)\eta(\bxi))^p$ 
to the right 
using repeatedly Propositions \ref{product:prop} 
and \ref{sandwich:prop}:  
\begin{equation*}
\| \op_\a (b) 
g_p(T_\a(a)) - \op_\a (b) g_p(T_\a(\z \eta a))
\|_{\GS_1}\le C_p (\a\ell\rho)^{d-1}. 
\end{equation*} 
The same bound holds if $T_\a$ is replaced with $S_\a$. 
Now, commuting $\z\eta a$ to the left, with the help of 
Propositions \ref{product:prop} and \ref{sandwich:prop} again 
we arrive at \eqref{through:eq}. 
\end{proof}

\begin{prop}\label{HS:prop} 
Let $\L$, $\Om$ be basic Lipschitz domains. 
Suppose that the symbol $b\in\BS^{(d+2, d+2)}$ satisfies \eqref{support:eq}, 
and that $\a\ell\rho\ge 2$. Then 
\begin{equation}\label{HSestim:eq}
\|\chi_{\L} \op_\a (b) P_{\Om, \a} (1-\chi_{\L})\|_{\GS_1}
\le C(\a\ell\rho)^{d-1} \log(\a\ell\rho)
 \SN^{(d+2, d+2)}(b; \ell, \rho) . 
\end{equation}
\end{prop}

See \cite{Sob_2013}, Theorem 4.6.

Here is a useful consequence of the above bound:

\begin{cor}
Under the conditions of Proposition 
\ref{HS:prop},
\begin{equation}\label{nucl:eq}
\|\op_\a(b) T_\a(1)\bigl(
I - T_\a(1)\bigr)\|_{\GS_1}\le C(\a\ell\rho)^{d-1} \log(\a\ell\rho)\ 
\SN^{(d+2, d+2)}(b; \ell, \rho).
\end{equation} 
\end{cor}

\begin{proof} 
Without loss of generality assume that $\SN^{(d+2, d+2)}(b; \ell, \rho)=1$. 
Calculate:
\begin{equation*}
T_\a(1)\bigl(I-T_\a(1)\bigr) 
= \chi_\L P_{\Om, \a}\bigl(1-\chi_\L\bigr)P_{\Om, \a}\chi_\L,
\end{equation*}
so that 
\begin{equation*}
\|\op_\a(b) T_\a(1)\bigl(
I - T_\a(1)\bigr)\|_{\GS_1}
\le \| [\op_\a(b), \chi_\L]\|_{\GS_1}
+ \| \chi_\L \op_\a(b)P_{\Om, \a} (1-\chi_\L)\|_{\GS_1}. 
\end{equation*}
Propositions \ref{sandwich:prop} and \ref{HS:prop} 
lead to \eqref{nucl:eq}.  
\end{proof}

\subsection{Bounds and asymptotics for more general domains} 
The next group of results expresses the fact that 
the local asymptotics are 
determined by local properties of the boundaries $\p\L, \p\Om$.
This is the key idea in the proof of Theorem \ref{main:thm}. 
Let $\L$, $\Om$ and $\L_0, \Om_0$ be two pairs of domains such 
that each of $\L_0, \Om_0$ is either 
\begin{enumerate}
\item 
a basic Lipschitz domain, or
\item
the entire space $\R^d$, or
\item
the empty set. 
\end{enumerate}
Suppose that   
\begin{equation}\label{localdom:eq}
\L\cap B(\bz, \ell) = \L_0\cap B(\bz, \ell),\ \ 
\Om\cap B(\bmu, \rho) = \Om_0\cap B(\bmu, \rho).
\end{equation}
%
The next localization result is crucial.

\begin{lem}\label{localization:lem}
Let $a, b\in\BS^{(d+2, d+2)}$, and let $b$ satisfy \eqref{support:eq}. 
Suppose that the domains $\L$, $\Om$ and $\L_0, \Om_0$  
are as specified above, and that $\a\ell\rho\ge \ell_0$. Then 
\begin{equation}\label{localcom:eq}
\|[\op_\a(b), P_{\Om, \a}]\|_{\GS_1} +
\|[\op_\a(b), \chi_\L]\|_{\GS_1}\le C  (\a\ell\rho)^{d-1} 
\SN^{(d+2, d+2)}(b; \ell, \rho),
\end{equation}
and 
\begin{align}\label{localization:eq}
\| \op_\a(b)\bigl( 
g_p(T_\a(a; \L, \Om))
- &\ g_p(T_\a(a; \L_0, \Om_0))\bigr)\|_{\GS_1}\notag\\[0.2cm]
\le &\ C(\a\ell\rho)^{d-1}
\SN^{(d+2, d+2)}(b; \ell, \rho)\bigl(\SN^{(d+2, d+2)}(a; \ell, \rho)\bigr)^p.
\end{align}
The same bound holds if $T_\a$ is replaced with $S_\a$. 
\end{lem}

For $\plainC1$-domains $\L, \Om$ estimates of this type  were established 
in \cite{Sob}, Section 7. 
Generalization to the Lipschitz domains is quite straightforward  
and we present a proof here  
for the sake of completeness.

\begin{proof}[Proof of Lemma \ref{localization:lem}]
Without loss of generality assume that 
the both $\SN$-norms on the right-hand sides of 
\eqref{localcom:eq} and \eqref{localization:eq} 
equal $1$. 
For any two operators $A_1$ and $A_2$ 
we write $A_1\thicksim A_2$ if $\|A_1-A_2\|_{\GS_1}\le C(\a\ell\rho)^{d-1}$, 
with a constant $C$ independent of $\a, \ell, \rho$.

Assume that $\L_0$, $\Om_0$ are basic Lipschitz domains. 
The following relations are consequences of \eqref{lr:eq} 
and  Proposition \ref{sandwich:prop}:
\begin{equation}\label{lambdacom:eq}
\op_\a(b) \chi_\L \thicksim 
\op_\a^{\rm r}(b)\chi_{\L_0}
\thicksim \chi_{\L_0}\op_\a^{\rm r}(b)
\thicksim \chi_{\L_0}\op_\a(b).
\end{equation} 
Taking the adjoints we also get 
$\chi_\L\op_\a(b)\thicksim \op_\a(b)\chi_{\L_0}$. In the same way 
one obtains similar relations for $P_{\Om, \a}$:
\begin{equation}\label{omegacom:eq}
\op_\a(b) P_{\Omega, \a}
\thicksim P_{\Om_0, \a}\op_\a(b),\ 
 P_{\Omega, \a} \op_\a(b)
\thicksim \op_\a(b) P_{\Om_0, \a}. 
\end{equation}  
Thus by Proposition \ref{sandwich:prop},
\begin{equation*}
\begin{cases}
[\op_\a(b), \chi_\L] \thicksim [\op_{\a}(b), \chi_{\L_0}] \thicksim 0,\\[0.2cm]
[\op_\a(b), P_{\Om, \a}] \thicksim [\op_{\a}(b), P_{\Om_0, \a}] \thicksim 0.
\end{cases}
\end{equation*}
If $\L_0$ or $\Om_0$  
are either $\R^d$ or $\varnothing$, then the above relations hold 
for trivial reasons. 
This proves \eqref{localcom:eq}. 

Applying repeatedly 
the relations \eqref{lambdacom:eq} and \eqref{omegacom:eq} 
in combination with Proposition \ref{product:prop} we arrive at 
\begin{equation*}
\op_\a(b) \bigl(T_\a(a; \L, \Om)\bigr)^p 
\thicksim \bigl(T_\a(a; \L_0, \Om_0)\bigr)^p \op_\a(b)
\thicksim \op_\a(b)\bigl(T_\a(a; \L_0, \Om_0)\bigr)^p.
\end{equation*}
This relation coincides with \eqref{localization:eq}. 

The same argument leads to the bound of the form 
\eqref{localization:eq} for the operator $S_\a$.
\end{proof} 

\begin{lem}\label{localasymptotics:lem}
Let $a, b\in\BS^{(d+2, d+2)}$,  
and assume that $b$ satisfies 
\eqref{support:eq}. 
Let $\a\ell\rho\ge \ell_0$. 
Suppose that $\L$ and $\Om$ satisfy 
\eqref{localdom:eq}, and 
one of the following two conditions is satisfied:
\begin{enumerate}
\item
$\L_0 = \varnothing$ or $\L_0 = \R^d$,
\item
$\Om_0 = \varnothing$ or $\Om_0 = \R^d$.
\end{enumerate}
Then 
\begin{align}\label{asymptotics1:eq}
\bigl|\tr\bigl(\op_\a(b) g_p(T_\a(a))\bigr)
- &\ \a^d \GW_0(b g_p(a))|\notag\\[0.2cm]
\le &\ C_p(\a\ell\rho)^{d-1} 
\SN^{(d+2, d+2)}(b; \ell, \rho)\bigl(\SN^{(d+2, d+2)}(a; \ell, \rho)\bigr)^p.
\end{align}
 \end{lem}

%

\begin{proof} 
By Lemmas \ref{localization:lem} and \ref{through:lem} 
we may assume that $\L = \L_0, \Om = \Om_0$ and $a \equiv 1$. 
Under any of the conditions of the lemma 
we have either $T_\a(1; \L_0, \Om_0) = 0$ or $\chi_{\L_0}$ or $P_{\Om_0,\a}$. 
In the first case the left-hand side of \eqref{asymptotics1:eq} 
equals zero, and there is nothing to prove. 
If $T_\a(1) = \chi_{\L_0}$, then the sought trace has the form 
$\tr (\op_\a(b) \chi_{\L_0})$. 
This trace is easily found by integrating the kernel of the operator 
over the diagonal, and it does not depend on the choice of quantization. 
This immediately leads to \eqref{asymptotics1:eq}. 
If $T_\a(1) = P_{\Om_0, \a}$, then computing the trace 
$\tr (\op_\a^{\rm l}(b\chi_{\Om_0}))$ we obtain \eqref{asymptotics1:eq} 
again. Note that in this case it is convenient to choose the l-quantization 
for $\op_\a(b)$.  
\end{proof}

The next result is also useful. 

\begin{lem}\label{semi:lem}
Let the symbols $a$, $b$ be as in Lemma \ref{localasymptotics:lem}, and let 
$\a\ell\rho\ge \ell_0$. Suppose that $\L$ and $\Om$ satisfy \eqref{localdom:eq}. 
Then 
\begin{align}\label{asymptotics2:eq}
\bigl|\tr\bigl(\op_\a(b) \chi_\L 
g_p(T_\a(a; \R^d, \Om))\chi_\L\bigr)
- &\ \a^d \GW_0(b g_p(a); \L, \Om)|\notag\\[0.2cm]
\le &\ C_p(\a\ell\rho)^{d-1} 
\SN^{(d+2, d+2)}(b; \ell, \rho)\bigl(\SN^{(d+2, d+2)}(a; \ell, \rho)\bigr)^p.
\end{align}
\end{lem}

\begin{proof}
Due to \eqref{localcom:eq}, 
the problem reduces to finding the trace of the operator 
\begin{equation*}
\chi_\L\op_\a^{\rm l}(b) 
g_p(T_\a(a; \R^d, \Om))\chi_\L.
\end{equation*}
As in the proof of Lemma \ref{localasymptotics:lem}, by virtue of 
Lemmas \ref{localization:lem} and \ref{through:lem}  
we may assume that $\Om = \Om_0$ 
and $a\equiv 1$. Thus $T_\a(1; \R^d, \Om_0) = P_{\Om_0, \a}$. Again, 
the trace of the operator 
$\chi_\L\op_\a^{\rm l}(b) P_{\Om_0, \a} \chi_\L$ is easily seen to be equal to 
$\a^d \GW_0(b; \L, \Om)$.
\end{proof}

So far it was enough to assume that the domains were Lipschitz. To 
state the asymptotic result we need more restrictive conditions.  
 
\begin{prop}\label{sobw:prop}
Let $a, b\in\BS^{(d+2, d+2)}$, and let $b$ satisfy \eqref{support:eq}. 
 Assume that \eqref{localdom:eq} holds with some basic domains $\L_0, \Om_0$ 
 such that $\L_0$ is $\plainC1$ and $\Om_0$ is $\plainC3$.  
Then 
\begin{align}\label{toeplitz:eq}
\tr \bigl(\op_\a^{\rm l}(b) g_p(T_\a(a))\bigr)
= &\ \a^d \GW_0(b g_p(a); \L, \Om)\notag\\[0.2cm]
&\ + \a^{d-1} \log\a\  \GW_1(b\GA(g_p; a); \p\L, \p\Om) + o(\a^{d-1}\log\a),
\end{align}
as $\a\to\infty$.
\end{prop}

This proposition follows from \cite{Sob}, Theorem 11.1 upon application of 
Lemma \ref{localization:lem}.  

\section{Proof of the main theorems}
  
Here we concentrate on proving Theorems \ref{main:thm} and \ref{main_inf:thm}. 
As explained earlier, Theorem \ref{main:thm} implies the main results -- 
Theorems \ref{main_anal:thm} and \ref{main_s:thm}. 

\subsection{An intermediate local asymptotics}
We begin with the following local result:

\begin{thm}\label{main_local:thm}  
Suppose that $b\in\BS^{(d+2, d+2)}$  
is a symbol with compact support in both variables, and that 
$\L$ is 
a piece-wise $\plainC1$ basic domain, and 
$\Om$ a piece-wise $\plainC3$ basic domain. 
Then 
\begin{align}\label{toeplitza1:eq}
\tr \bigl(\op_\a^{\rm l}(b) g_p(T_\a(1))\bigr)
= &\ \a^d \GW_0(b; \L, \Om)\notag\\[0.2cm]
&\ + \a^{d-1} \log\a\  \GW_1(b\GA(g_p; 1); \p\L, \p\Om) + o(\a^{d-1}\log\a),
\end{align}
as $\a\to\infty$. 
\end{thm}

Without loss of generality assume that 
the symbol $b$ is supported on $B(\bold0, 1)\times B(\bold0, 1)$. 
If $B(\bold0, 1)\cap\p\L = \varnothing$ or $B(\bold0, 1)\cap\p\Om = \varnothing$, 
then the required asymptotics immediately follow from Lemma \ref{localasymptotics:lem}. 
Assume that neither of the above intersections is empty. Cover 
the boundaries $\p\L\cap B(\bold0, 1)$ and 
$\p\Om\cap B(\bold0, 1)$ with finitely many open balls 
of radius $\varepsilon>0$. Denote the number of such balls 
by $J=J_\varepsilon$ and $K = K_\varepsilon$ respectively. 
Since $\p\L$ and $\p\Om$ are Lipschitz, 
one can construct these coverings in such a way that the number of intersections 
of  each ball with the other ones is bounded from above 
uniformly in $\varepsilon$ and 
\begin{equation}\label{pokr:eq}
J_\varepsilon,\ K_\varepsilon\le C\varepsilon^{1-d}.
\end{equation} 
Let $\Sigma_\L$ (resp. $\Sigma_\Om$) 
be the set of indices $j$ (resp. $k$) such that 
the ball from the constructed covering 
indexed $j$ (resp. $k$) has a non-empty intersection  
with the set $(\p\L)_{\rm s}$ (resp. $(\p\Om)_{\rm s}$). 
Since the sets $(\p\L)_{\rm s}, (\p\Om)_{\rm s}$ 
are built out of Lipschitz surfaces, by construction of the covering we have  
\begin{equation}\label{no_sigma:eq}
\#(\Sigma_\L), \#(\Sigma_\Om)\le C\varepsilon^{2-d}. 
\end{equation}
We may assume that the covering balls with indices $j\notin\Sigma_\L$ (resp. $k\notin\Sigma_\Om$) 
are separated from $(\p\L)_{\rm s}$ (resp. $(\p\Om)_s$). Thus in each of these balls 
the boundary $\p\L$ (resp. $\p\Om$) is $\plainC1$ (resp. $\plainC3$).

Denote by $\phi_j, j=1, 2, \dots, J,$ 
and $\psi_k, k=1, 2,\dots, K$ the associated smooth partitions of unity, so 
that the functions 
\begin{equation*}
\phi(\bx):=
\sum_{j=1}^J \phi_j(\bx), \ \ 
\ \ 
\psi(\bxi):=
\sum_{j=1}^K \psi_k(\bxi)
\end{equation*}
equal $1$ on a neighbourhood of $\p\L\cap B(\bold0, 1)$ and 
$\p\Om\cap B(\bold0, 1)$ respectively, 
and 
\begin{equation}\label{partuni:eq}
|\nabla^n\phi_j(\bx)| + |\nabla^n\psi_j(\bxi)|\le C_n\varepsilon^{-n},\ n = 0, 1, \dots, 
\end{equation}
uniformly in $\bx$ and $\bxi$. 
The symbol $b(1-\phi\psi)$ is supported away 
from  $\bigl(\p\L\cap B(\bold0, 1)\bigr)
\times \bigl(\p\Om\cap B(\bold0, 1)\bigr)$. 
Thus Lemma \ref{localasymptotics:lem} implies that  
\begin{equation}\label{inside:eq}
\biggl|
\tr \bigl(\op^{\rm l}_\a(b(1-\phi\psi)) g_p(T_\a(1))\bigr) 
- \a^d \GW_0\bigl(b(1-\phi\psi)\bigr)
\biggr|\le C_\varepsilon \a^{d-1}.
\end{equation}
The constant $C_\varepsilon$ on the right-hand side depends 
on the symbol $b$, and on $\varepsilon$, but 
the latter fact does not matter for the rest of the proof. 
It remains to study the trace 
$\tr \bigl(\op^{\rm l}_\a(b\phi\psi) g_p(T_\a(1))\bigr)$. 

Let us separate contributions from the smooth and singular 
parts of the boundaries $\p\L$ and $\p\Om$. Denote 
\begin{equation*}
\tilde b(\bx, \bxi) = \sum_{j\notin\Sigma_\L} 
\sum_{k\notin \Sigma_\Om} b_{jk}(\bx, \bxi),\ 
b_{jk}(\bx, \bxi)
= 
b(\bx, \bxi) \phi_j(\bx) \psi_k(\bxi). 
\end{equation*} 
The support of $\tilde b$ contains only smooth 
parts of the boundaries $\p\L$ and $\p\Om$, so  
by Proposition \ref{sobw:prop}  we have 
\begin{align}
\lim_{\a\to\infty}\frac{1}{\a^{d-1}\log\a}
\biggl(\tr\bigl(\op_\a^{\rm l}(\tilde b) g_p(T_\a(1))\bigr)
- &\ \a^d \GW_0\bigl(\tilde b \bigr)\notag\\[0.2cm]
&\ - \a^{d-1}\log\a\  \GW_1\bigl(\tilde b \GA(g_p; 1)\bigr)\biggr) = 0.
\label{last_step:eq}
\end{align}
It remains to handle  the cases when 
$j\in\Sigma_\L$ or $k\in\Sigma_\Om$.  
Let 
\begin{equation*}
\boldsymbol\Sigma = \{(j, k): j\in\Sigma_\L \ \textup{or}\ \ k\in \Sigma_\Om\}.
\end{equation*}

\begin{lem} 
Let $b_{jk}$ be as defined above, and let $p\ge 1$. Then 
\begin{align}
\limsup\frac{1}{\a^{d-1} \log\a}\sum_{(j, k)\in\boldsymbol\Sigma}&\  
\biggl|\tr \bigl(\op^{\rm l}_\a(b_{jk}) g_p(T_\a(1))\bigr)\notag \\[0,3cm]
- &\ \a^{d} \GW_0(b_{jk})
- \a^{d-1} \log\a \ \GW_1(b_{jk} \GA(g_p; 1))\biggr|
\le C\varepsilon,\label{sum_jk:eq}
\end{align}
as $\a\to\infty$.
\end{lem}

\begin{proof}
It is enough to establish the estimate 
\begin{align}
\limsup_{\a\to\infty}
\frac{1}{\a^{d-1} \log\a}&\  
\biggl|\tr \bigl(\op^{\rm l}_\a(b_{jk}) g_p(T_\a(1))\bigr) \notag\\[0,3cm]
- &\ \a^{d} \GW_0(b_{jk})
- \a^{d-1} \log\a \ \GW_1(b_{jk} \GA(g_p; 1))\biggr|
\le C\varepsilon^{2(d-1)}.\label{jk:eq}
\end{align}
Indeed, in view of \eqref{pokr:eq} and \eqref{no_sigma:eq}, 
the number of summands on the left-hand side of 
\eqref{sum_jk:eq} does not exceed  $C\varepsilon^{3-2d}$, and hence 
summing \eqref{jk:eq} up over $(j, k)\in\boldsymbol\Sigma$ we obtain 
\eqref{sum_jk:eq}.  
If $p = 1$, then the trace asymptotics of the 
operator $\op_\a^{\rm l}(b_{jk}) T_\a(1)$ are easy to find. 
Indeed, by Lemma \ref{semi:lem}, we have 
\begin{equation}\label{linear:eq}
 \tr \bigl(\op^{\rm l}_\a (b_{jk})\chi_\L P_{\Om, \a}\chi_\L\bigr)
 = \a^d \GW_0(b_{jk}) + O((\a\varepsilon^2)^{d-1}). 
\end{equation}
Thus it remains to study the trace of the operator
\begin{equation*}
\op^{\rm l}_\a( b_{jk}) g(T_\a(1)), \ g(t) = t^p-t,
\end{equation*}
with $p\ge 2$.
Represent $g(t) = t(1-t) \tilde g(t)$ with a polynomial $\tilde g$,  
and estimate using \eqref{nucl:eq}: 
\begin{align*}
\|\op^{\rm l}_\a(b_{jk}) g(T_\a(1))\|_{\GS_1}
\le &\ \|\op^{\rm l}_\a(b_{jk}) T_\a(1) \bigl(I-T_\a(1)\bigr)\|_{\GS_1} \| \tilde g(T_\a(1))\|\\[0.2cm]
\le &\ 
C \max_{0\le t\le 1}|\tilde g(t)| 
(\a\varepsilon^2)^{d-1}\log(\a\varepsilon^2), 
\end{align*}
for sufficiently large $\a$. 
Together with \eqref{linear:eq} this implies that 
\begin{equation}\label{pre_jk:eq}
\limsup\frac{1}{\a^{d-1} \log\a}  
\biggl|\tr \bigl(\op^{\rm l}_\a(b_{jk}) g_p(T_\a(1))\bigr) 
- \a^{d} \GW_0( b_{jk})\biggr|
\le C\varepsilon^{2(d-1)}, 
\end{equation}
as $\a\to\infty$. 
It follows straight from the definition 
\eqref{w1:eq} that  
\begin{equation*}
\bigl|\GW_1\bigl(b_{jk}\GA(g_p; 1)\bigr)\bigr|
\le C\varepsilon^{2(d-1)}, 
\end{equation*}
so \eqref{pre_jk:eq} entails \eqref{jk:eq}, as claimed. 
\end{proof}

\begin{proof}[Proof of Theorem 
\ref{main_local:thm}]
Remembering that $\GW_1(b(1-\phi\psi) \GA(g_p; 1)) = 0$, and 
putting together \eqref{last_step:eq}, \eqref{inside:eq} and  
  \eqref{sum_jk:eq} we obtain that 
 \begin{equation*}
\limsup\frac{1}{\a^{d-1} \log\a}
\biggl|\tr \bigl( \op^{\rm l}_\a(b)g_p(T_\a(1))\bigr) 
- \a^{d} \GW_0( b)
- \a^{d-1} \log\a \ \GW_1\bigl(b \GA(g_p; 1)\bigr)\biggr|
\le C\varepsilon, 
\end{equation*} 
as $\a\to\infty$, for any $\varepsilon>0$. Taking $\varepsilon\to 0$ 
we arrive at the asymptotics \eqref{toeplitza1:eq}. 
\end{proof}

\subsection{Proof of Theorems \ref{main:thm} 
and \ref{main_inf:thm}}  
The proofs amount to putting together local asymptotic formulas and 
estimates obtained above. The argument is based on partition of unity, and 
is rather standard. We present it for the sake of completeness. 
Also, all the proofs are conducted for the operator $T_\a$ only
-- the argument  
for $S_\a$ is essentially the same. 
  
The next two lemmas are the last building blocks in the 
proofs of Theorems \ref{main:thm} and \ref{main_inf:thm}. 

\begin{lem} 
Let the conditions of Theorem \ref{main_inf:thm} 
be satisfied. 
Let $h\in\plainC\infty_0(\R^d)$ be an arbitrary function. 
Then 
\begin{align}\label{main_inf_h:eq}
\lim_{\a\to\infty}\frac{1}{\a^{d-1}\log\a}
\bigl[
\tr \bigl(h g_p(T_\a(a; \L, \Om))\bigr) -  \a^d \GW_0(h &\ g_p(a); \L, \Om)
\bigr]\notag\\[0.2cm]
=  &\ \GW_1(\GA(h g_p; a); \p\L, \p\Om),
\end{align} 
and 
\begin{align}\label{main_inf_h1:eq}
\lim_{\a\to\infty}\frac{1}{\a^{d-1}\log\a}
\tr\bigl[\tr \bigl(h \chi_\L g_p(T_\a(a; \R^d, \Om))\chi_\L\bigr) 
- \a^d \GW_0(h g_p(a); \L, \Om)
\bigr] = 0.
\end{align} 
If $T_\a(a)$ is replaced with $S_\a(a)$, then the same formulas hold 
with the symbol $a$ replaced by $\re a$ 
in $\GW_0$ and $\GW_1$. 
\end{lem}
  
\begin{proof} 
Let $R>0$ be such that $\supp h\in B(\bold0, R)$, and 
either $\L$ or $\R^d\setminus\L$ 
is contained in $B(\bold0, R)$. 
Since the domains $\L\cap B(\bold0, R)$ 
and $\Om$ are bounded, we can cover their closures 
by finitely many open 
balls such that in each of them each domain 
$\L$ or $\Om$ is represented by a basic domain or by $\R^d$. 
Denote by $\{\phi_j\}$ 
and  $\{\psi_k\}$ the partitions of unity subordinate to these coverings. 
Represent 
\begin{equation*}
h\chi_\L P_{\Om, \a} = \sum_{j,k} \chi_\L \op_\a^{\rm l}(b_{jk})P_{\Om, \a},\ 
b_{jk}(\bx, \bxi) = h(\bx) \phi_j(\bx) \psi_k(\bxi).
\end{equation*}
Consequently, in order to prove \eqref{main_inf_h:eq} 
it suffices to find the sought asymptotics 
for the operator 
\begin{equation*}
 \chi_\L \op_\a^{\rm l}(b_{jk}) 
 P_{\Om, \a}\op_\a^{\rm l}(a) P_{\Om, \a}(T_\a(a; \L, \Om))^{p-1},
\end{equation*}
for each $j$ and $k$. 
By virtue of \eqref{localcom:eq} this is equivalent to studying the operator 
\begin{equation*}
\op_\a^{\rm l}(b_{jk}) \bigl(T_\a(a; \L, \Om)\bigr)^p. 
\end{equation*}
Now, due to \eqref{localization:eq}, we can replace each 
domain $\L$ or $\Om$ by the 
appropriate basic domain or by $\R^d$. 
Furthermore,  Lemma \ref{through:lem} ensures that the symbol $a$ can be 
replaced by the constant symbol $a \equiv 1$. 
Now Theorem \ref{main_local:thm} implies that 
\begin{align}\label{po:eq}
\tr \op_\a^{\rm l}(b_{jk}) g_p\bigl(T_\a(1; \L, \Om)\bigr)
= &\ \a^d \GW_0(b_{jk}; \L, \Om)\notag\\[0.2cm]
&\ + \a^{d-1} \log\a\  \GW_1(b_{jk}\GA(g_p; 1); \p\L, \p\Om) + o(\a^{d-1}\log\a),
\end{align}
as $\a\to\infty$. 
Summing over $j$ and $k$ we obtain formula \eqref{main_inf_h:eq}.  

Similarly, 
for the asymptotics \eqref{main_inf_h1:eq} it suffices to study the operator 
\begin{equation*}
\op_\a^{\rm l}(b_{jk}) \chi_\L\bigl(T_\a(a; \R^d, \Om)\bigr)^p\chi_\L. 
\end{equation*}
By Lemma \ref{semi:lem}, the trace of this operator equals 
$\a^{d}\GW_0(b_{jk} g_p(a); \L, \Om)+ O(\a^{d-1})$.
Summing over $j$ and $k$ we obtain formula \eqref{main_inf_h1:eq}, as required. 
\end{proof}  
   
The following lemma concentrates on the case 
of unbounded $\L$.

\begin{lem}\label{DA:lem} 
Suppose that $a\in \BS^{(d+2, d+2)}$. 
Let $\Om$ and $\L$ be Lipschitz domains 
such that $\Om$ and $\R^d\setminus\L$ are bounded. 
Let $h\in\plainC\infty_0(\R^d)$ be 
a function such that $h(\bx) = 1$ for all $\bx\in\R^d\setminus\L$. Then 
\begin{align}\label{DA:eq}
\|(I-h) \bigl[
g_p\bigl(T_\a(a, \L, \Om)\bigr) 
- g_p\bigl( T_\a(a, \R^d, \Om)\bigr)
\bigr]\|_{\GS_1}
\le &\ C \a^{d-1},
\end{align}
for any $p = 1, 2, \dots$, 
with a constant $C$ independent of $\a$. The same bound holds 
if $T_\a$ is replaced with $S_\a$. 
\end{lem}

\begin{proof}  
For brevity we write $T_\a(\L) = T_\a(a; \L, \Om)$. 
For any two operators $A_1$ and $A_2$ 
we write $A_1\thicksim A_2$ if $\|A_1-A_2\|_{\GS_1}\le C\a^{d-1}$, 
with a constant $C$ independent of $\a$.

First we prove that 
\begin{align}\label{infcomm:eq}
(I-h) (T_\a(\L))^p \sim (T_\a(\R^d))^p(I-h).
\end{align}
Suppose that $p = 1$. 
Let $\eta\in\plainC\infty_0(\R^d)$ be a function 
such that $\eta\chi_\Om = \chi_\Om$, and let $b(\bx, \bxi) = h(\bx) \eta(\bxi)$.   
Since $(1-h)\chi_\L = 1-h$, we have 
\begin{equation}\label{perep:eq}
(I-h) T_\a(\L) = T_\a(\R^d)\chi_\L - \op_\a^{\rm l}(b) T_\a(\R^d) \chi_\L.  
\end{equation}
Using the partition of unity $\{\psi_j\}$ featuring in the 
proof of the previous lemma, and then bound \eqref{localcom:eq} 
and Lemma \ref{product:prop}, we can claim that  
\begin{equation*}
\op_\a^{\rm l}(b) T_\a(\R^d) \chi_\L\sim
 T_\a(\R^d) \op_\a^{\rm l}(b)\chi_\L 
 \sim  T_\a(\R^d) \op_\a^{\rm r}(b)\chi_\L
= T_\a(\R^d) h \chi_\L.
\end{equation*} 
Together with \eqref{perep:eq} this gives \eqref{infcomm:eq} for $p=1$.

Suppose now that \eqref{infcomm:eq} holds for $p= k$, and let us 
prove it for $p = k+1$. Write:
\begin{align*}
(I-h) (T_\a(\L))^{k+1} - &\ (T_\a(\R^d))^{k+1}(I-h)\\[0.2cm]
= &\ \bigl[(I-h) (T_\a(\L))^k 
-  (T_\a(\R^d))^k(I-h)\bigr] T_\a(\L) \\[0.2cm]
&\ + 
(T_\a(\R^d))^k\bigl[(I-h) T_\a(\L) - T_\a(\R^d)(I-h)\bigr].
\end{align*} 
The sought bound follows from \eqref{infcomm:eq} for $p=1$ and $p=k$.

To conclude the proof write
\begin{align*}
(I-h) \bigl[(T_\a(\L))^{p} - &\ (T_\a(\R^d))^{p}\bigr]\\[0.2cm]
 = &\ (I-h) (T_\a(\L))^p -  (T_\a(\R^d))^p(I-h)\\[0.2cm]
&\ + \bigl[ \bigl(T_\a(\R^d)\bigr)^p, I-h \bigr],
\end{align*}
so that \eqref{DA:eq} follows from \eqref{infcomm:eq} used twice: 
for the domain $\L$ itself, and for $\L = \R^d$.
\end{proof}

Now we can proceed to the proof of Theorems \ref{main:thm} 
and \ref{main_inf:thm}. As explained earlier, the proofs are conducted only 
for the operators $T_\a$.

\begin{proof} 
[Proof of Theorem \ref{main:thm}] 
Since $\L$ is bounded, use formula \eqref{main_inf_h:eq} 
with a function $h\in\plainC\infty_0(\R^d)$ such that $h\chi_\L = \chi_\L$. 
This completes the proof. 
\end{proof}

\begin{proof}[Proof of Theorem \ref{main_inf:thm}] 
If $\L$ is bounded, then Theorem \ref{main_inf:thm} 
follows from formulas \eqref{main_inf_h:eq} 
and \eqref{main_inf_h1:eq} with a function $h$ as in the above proof 
of Theorem \ref{main:thm}.

Suppose that $\R^d\setminus\L$ is bounded. 
%
%
%
Let $h\in\plainC\infty_0(\R^d)$ 
be a function such that $h(\bx) = 1$ for $\bx\in \R^d\setminus \L$. 
Due to Lemma 
\ref{DA:lem} it suffices to 
establish the formula
\begin{align*}
\tr \bigl\{h\bigl[ 
g_p(T_\a(\L)) - \chi_\L &\ g_p(T_\a(\R^d))\chi_\L
\bigr]\bigr\}\notag \\[0.2cm]
= &\ \a^{d-1}\log\a\  \GW_1(\GA(g_p; a); \L, \Om) + o(\a^{d-1}\log\a),
\end{align*}
where we have denoted $T_\a(\L) = T_\a(a; \L, \Om)$.
But this formula immediately follows from 
\eqref{main_inf_h:eq} and \eqref{main_inf_h1:eq} again. 
Thus the proof is complete.
\end{proof}

Theorems~\ref{main_anal:thm} and \ref{main_s:thm} are 
derived from Theorem~\ref{main:thm} 
by approximating $g$ with polynomials, in the same way as in \cite{Sob}, Section 12.


\bibliographystyle{amsplain}

\begin{thebibliography}{99}

\bibitem{Cordes} H.O. Cordes, \emph{On compactness of commutators of multiplications and
convolutions, and boundedness of pseudodifferential operators}, J. Funct. Anal.
\textbf{18} (1975), 115--131.

\bibitem{GiKl} D. Gioev, I. Klich,
\emph{Entanglement Entropy of fermions in any dimension and the Widom Conjecture},
Phys. Rev. Lett. \textbf{96} (2006), no. 10, 100503, 4pp.

\bibitem{HLS}R.C. Helling, H. Leschke, W.L. Spitzer,
\emph{
A special case of a conjecture by widom with implications
to fermionic entanglement entropy}, 
Int. Math. Res. Notices vol. 2011 (2011), pp 1451-1482.

\bibitem{LSS} 
H. Leschke, W.L. Spitzer, A. V. Sobolev, 
\emph{Scaling of R\'enyi entanglement entropies of the free Fermi-gas ground
state: A rigorous proof}, 
Phys. Rev. Lett. \textbf{112}, 160403.

\bibitem{Rob}
D. Robert, \emph{Autour de l'Approximation 
Semi-Classique}, Progress in Mathematics, Birkh\"auser, Boston, 1987.   

\bibitem{Sob}
A.V. Sobolev, \emph{Pseudo-differential operators 
with discontinuous symbols: Widom's conjecture}, Mem. AMS \textbf{222} (2013), no 1043.  


\bibitem{Sob_2013} A. V. Sobolev,
\emph{On the Schatten-von Neumann 
properties of some pseudo-differential operators}, 
Journal of Functional Analysis \textbf{266} (2014), 5886--5911. 

\bibitem{Widom_82} H. Widom,
\emph{On a class of integral
operators with discontinuous symbol},
Toeplitz centennial (Tel Aviv, 1981), pp. 477--500,
Operator Theory: Adv. Appl., 4, Birkh\"auser, Basel-Boston, Mass., 1982.

\bibitem{Widom_85} 
H. Widom, \emph{Asymptotic expansions for pseudodifferential operators 
on bounded domains}, Lecture Notes in Mathematics, V. 1152, Springer, 1985. 
 
 
\bibitem{Widom_90} H. Widom,
\emph{On a class of integral operators on a half-space
with discontinuous symbol},  J. Funct. Anal.  \textbf{88}
(1990),  no. 1, 166--193. 

\end{thebibliography}

\end{document}